\documentclass[a4paper,english,fontsize=11pt,parskip=half,abstracton]{scrartcl}
\usepackage{babel}
\usepackage[utf8]{inputenc}
\usepackage[T1]{fontenc}
\usepackage[a4paper,left=20mm,right=20mm,top=30mm,bottom=30mm]{geometry}
\usepackage{amsmath}
\usepackage{amsthm}
\usepackage{amssymb}
\usepackage{enumerate}
\usepackage{aliascnt}
\usepackage[bookmarks=false,
            pdftitle={Solution of Brauer's k(B)-Conjecture for pi-blocks of pi-separable groups},
            pdfauthor={Benjamin Sambale},
            pdfkeywords={},
            pdfstartview={FitH}]{hyperref}

\newtheorem{Thm}{Theorem} 
\newaliascnt{Lem}{Thm}
\newtheorem{Lem}[Lem]{Lemma}
\aliascntresetthe{Lem}
\newaliascnt{Prop}{Thm}

\aliascntresetthe{Prop}
\newaliascnt{Cor}{Thm}
\newtheorem{Cor}[Cor]{Corollary}
\aliascntresetthe{Cor}
\newaliascnt{Con}{Thm}
\newtheorem{Con}[Con]{Conjecture}
\aliascntresetthe{Con}
\theoremstyle{definition}
\newaliascnt{Def}{Thm}

\aliascntresetthe{Def}
\newaliascnt{Ex}{Thm}

\aliascntresetthe{Ex}

\numberwithin{equation}{section}

\allowdisplaybreaks[1]

\renewcommand{\phi}{\varphi}
\newcommand{\C}{\operatorname{C}}
\newcommand{\N}{\operatorname{N}}
\newcommand{\Z}{\operatorname{Z}}

\newcommand{\Aut}{\operatorname{Aut}}

\newcommand{\Out}{\operatorname{Out}}
\newcommand{\pcore}{\operatorname{O}}

\newcommand{\PSL}{\operatorname{PSL}}

\newcommand{\Irr}{\operatorname{Irr}}

\mathchardef\ordinarycolon\mathcode`\:  
 \mathcode`\:=\string"8000
 \begingroup \catcode`\:=\active
   \gdef:{\mathrel{\mathop\ordinarycolon}}
 \endgroup

\title{Solution of Brauer's $k(B)$-Conjecture for $\pi$-blocks of $\pi$-separable groups}
\author{Benjamin Sambale\footnote{Fachbereich Mathematik, TU Kaiserslautern, 67653 Kaiserslautern, Germany, 
\href{mailto:sambale@mathematik.uni-kl.de}{sambale@mathematik.uni-kl.de}}}
\date{\today}

\begin{document}
\frenchspacing
\maketitle
\begin{abstract}\noindent
Answering a question of Pálfy and Pyber, we first prove the following extension of the $k(GV)$-Problem: Let $G$ be a finite group and let $A$ be a coprime automorphism group of $G$. Then the number of conjugacy classes of the semidirect product $G\rtimes A$ is at most $|G|$. As a consequence we verify Brauer's $k(B)$-Conjecture for $\pi$-blocks of $\pi$-separable groups which was proposed by Y.~Liu. This generalizes the corresponding result for blocks of $p$-solvable groups. We also discuss equality in Brauer's Conjecture. On the other hand, we construct a counterexample to a version of Olsson's Conjecture for $\pi$-blocks which was also introduced by Liu.
\end{abstract}


\section{Introduction}

One of the oldest outstanding problems in the representation theory of finite groups is \emph{Brauer's $k(B)$-Conjecture}~\cite{Brauer46}. It asserts that the number $k(B)$ of ordinary irreducible characters in a $p$-block $B$ of a finite group $G$ is bounded by the order of a defect group of $B$. For $p$-solvable groups $G$, Nagao~\cite{NagaoKGV} has reduced Brauer's $k(B)$-Conjecture to the so-called \emph{$k(GV)$-Problem}: If a $p'$-group $G$ acts faithfully and irreducibly on a finite vector space $V$ in characteristic $p$, then the number $k(GV)$ of conjugacy classes of the semidirect product $G\ltimes V$ is at most $|V|$. 
Eventually, the $k(GV)$-Problem has been solved in 2004 by the combined effort of several mathematicians invoking the classification of the finite simple groups. A complete proof appeared in \cite{kGVproblem}. 

Brauer himself already tried to replace the prime $p$ in his theory by a set of primes $\pi$. Different approaches have been given later by Iizuka, Isaacs, Reynolds and others (see the references in \cite{Slattery}). Finally, Slattery developed in a series of papers~\cite{Slattery,Slattery2,Slattery3} a nice theory of $\pi$-blocks in $\pi$-separable groups (precise definitions are given in the third section below). This theory was later complemented by Laradji~\cite{Laradji,Laradji2} and Y. Zhu~\cite{Yixin2}. 
The success of this approach is emphasized by the verification of \emph{Brauer's Height Zero Conjecture} and the \emph{Alperin--McKay Conjecture} for $\pi$-blocks of $\pi$-separable groups by Manz--Staszewski~\cite[Theorem~3.3]{ManzStaszewski} and Wolf~\cite[Theorem~2.2]{Wolf} respectively.
In 2011, Y. Liu~\cite{Yanjun} put forward a variant of Brauer's $k(B)$-Conjecture for $\pi$-blocks in $\pi$-separable groups. Since $\{p\}$-separable groups are $p$-solvable and $\{p\}$-blocks are $p$-blocks, this generalizes the results mentioned in the first paragraph. 
Liu verified his conjecture in the special case where $G$ has a nilpotent normal Hall $\pi$-subgroup.
The aim of the present paper is to give a full proof of Brauer's $k(B)$-Conjecture for $\pi$-blocks in $\pi$-separable groups (see \autoref{main} below). In order to do so, we need to solve a generalization of the $k(GV)$-Problem (see \autoref{kGV} below). 
In this way we answer a question raised by Pálfy and Pyber at the end of \cite{PP} (see also \cite{INM}). The proof relies on the classification of the finite simple groups. 
Motivated by Robinson's theorem~\cite{Robinsonnonabel} for blocks of $p$-solvable groups, we also show that equality in Brauer's Conjecture can only occur for $\pi$-blocks with abelian defect groups.
Finally, we construct a counterexample to a version of \emph{Olsson's Conjecture} which was also proposed by Liu~\cite{Yanjun}. 

\section{A generalized $k(GV)$-Problem}

In the following we use the well-known formula $k(G)\le k(N)k(G/N)$ where $N\unlhd G$ (see \cite[Lemma~1]{NagaoKGV}).
\phantom{bla}

\begin{Thm}\label{kGV}
Let $G$ be a finite group, and let $A\le\Aut(G)$ such that $(|G|,|A|)=1$. Then $k(G\rtimes A)\le |G|$.
\end{Thm}
\begin{proof}
We argue by induction on $|G|$. The case $G=1$ is trivial and we may assume that $G\ne 1$. Suppose first that $G$ contains an $A$-invariant normal subgroup $N\unlhd G$ such that $1<N<G$. Let $B:=\C_A(G/N)\unlhd A$. Then $B$ acts faithfully on $N$ and by induction we obtain $k(NB)\le|N|$. 
Similarly we have $k((G/N)\rtimes(A/B))\le|G/N|$. It follows that
\[k(GA)\le k(NB)k(GA/NB)\le|N|k((G/N)(A/B))\le|N||G/N|=|G|.\]

Hence, we may assume that $G$ has no proper non-trivial $A$-invariant normal subgroups. In particular, $G$ is characteristically simple, i.\,e. $G=S_1\times\ldots\times S_n$ with simple groups $S:=S_1\cong\ldots\cong S_n$. If $S$ has prime order, then $G$ is elementary abelian and the claim follows from the solution of the $k(GV)$-Problem (see \cite{kGVproblem}). Therefore, we assume in the following that $S$ is non-abelian. 

We discuss the case $n=1$ (that is $G$ is simple) first. Since $(|A|,|G|)=1$, $A$ is isomorphic to a subgroup of $\Out(G)$.
If $G$ is an alternating group or a sporadic group, then $\lvert\Out(G)\rvert$ divides $4$ and $A=1$ as is well-known. In this case the claim follows since $k(GA)=k(G)\le|G|$. Hence, we may assume that $S$ is a group of Lie type over a field of size $p^f$ for a prime $p$.
According to the Atlas~\cite[Table~5]{Atlas}, the order of $\Out(G)$ has the form $dfg$.
Here $d$ divides the order of the Schur multiplier of $G$ and therefore every prime divisor of $d$ divides $|G|$. Moreover, $g\mid 6$ and in all cases $g$ divides $|G|$. Consequently, $|A|\le f\le\log_2p^f\le\log_2|G|$. On the other hand, \cite[Theorem~9]{RobRob} shows that $k(G)\le\sqrt{|G|}$. Altogether, we obtain 
\[k(GA)\le k(G)|A|\le\sqrt{|G|}\log_2|G|\le|G|\] 
(note that $|G|\ge |\mathfrak{A}_5|=60$ where $\mathfrak{A}_5$ denotes the alternating group of degree $5$).

It remains to handle the case $n>1$. 
Here $\Aut(G)\cong\Aut(S)\wr \mathfrak{S}_n$ where $\mathfrak{S}_n$ is the symmetric group of degree $n$. 
Let $B:=\N_A(S_1)\cap\ldots\cap\N_A(S_n)\unlhd A$. Then $B\le\Out(S_1)\times\ldots\times\Out(S_n)$ and the arguments from the $n=1$ case yield 
\begin{equation}\label{eq}
k(GB)\le k(G)|B|=k(S)^n|B|\le\bigl(\sqrt{|S|}\log_2|S|\bigr)^n.
\end{equation}
By Feit--Thompson, $|G|$ has even order and $A/B\le \mathfrak{S}_n$ has odd order since $(|G|,|A|)=1$. A theorem of Dixon~\cite{DixonTour} implies that $|A/B|\le \sqrt{3}^n$. 
If $|G|=60$, then $G\cong \mathfrak{A}_5$, $B=1$ and 
\[k(GA)\le k(\mathfrak{A}_5)^n|A|\le (5\sqrt{3})^n\le 60^n=|G|.\] 
Therefore, we may assume that $|G|\ge\lvert\PSL(3,2)\rvert=168$. Then \eqref{eq} gives
\[k(GA)\le k(GB)|A/B|\le(\sqrt{3|S|}\log_2|S|)^n\le |S|^n=|G|.\qedhere\]
\end{proof}

\section{$\pi$-Blocks of $\pi$-separable groups}

Let $\pi$ be a set of primes. Recall that a finite group $G$ is called \emph{$\pi$-separable} if $G$ has a normal series 
\[1=N_0\unlhd\ldots\unlhd N_k=G\] 
such that each quotient $N_i/N_{i-1}$ is a $\pi$-group or a $\pi'$-group. 
The following consequence of \autoref{kGV} generalizes and proves the conjecture made in \cite{INM}.

\begin{Cor}\label{cor}
For every $\pi$-separable group $G$ we have $k(G/\pcore_{\pi'}(G))\le|G|_{\pi}$. 
\end{Cor}
\begin{proof}
We may assume that $\pcore_{\pi'}(G)=1$ and $N:=\pcore_{\pi}(G)\ne 1$. We argue by induction on $|N|$. By the Schur--Zassenhaus Theorem, $N$ has a complement in $\pcore_{\pi\pi'}(G)$ and \autoref{kGV} implies $k(\pcore_{\pi\pi'}(G))\le|N|$. Now induction yields
\[k(G)\le k(\pcore_{\pi\pi'}(G))k(G/\pcore_{\pi\pi'}(G))\le|N||G/N|_\pi=|G|_\pi.\qedhere\]
\end{proof}

A \emph{$\pi$-block} of a $\pi$-separable group $G$ is a minimal non-empty subset $B\subseteq\Irr(B)$ such that $B$ is a union of $p$-blocks for every $p\in\pi$ (see \cite[Definition~1.12 and Theorem~2.15]{Slattery}). In particular, the $\{p\}$-blocks of $G$ are the $p$-blocks of $G$. In accordance with the notation for $p$-blocks we set $k(B):=|B|$ for every $\pi$-block $B$.

A \emph{defect group} $D$ of a $\pi$-block $B$ of $G$ is defined inductively as follows. Let $\chi\in B$ and let $\lambda\in\Irr(\pcore_{\pi'}(G))$ be a constituent of the restriction $\chi_{\pcore_{\pi'}(G)}$ (we say that $B$ \emph{lies over} $\lambda$). Let $G_\lambda$ be the inertial group of $\lambda$ in $G$. If $G_\lambda=G$, then $D$ is a Hall $\pi$-subgroup of $G$ (such subgroups always exist in $\pi$-separable groups). Otherwise we take a $\pi$-block $b$ of $G_\lambda$ lying over $\lambda$. Then $D$ is a defect group of $b$ up to $G$-conjugation (see \cite[Definition~2.2]{Slattery2}). It was shown in \cite[Theorem~2.1]{Slattery2} that this definition agrees with the usual definition for $p$-blocks.

The following theorem verifies Brauer's $k(B)$-Conjecture for $\pi$-blocks of $\pi$-separable groups (see \cite{Yanjun}).

\begin{Thm}\label{main}
Let $B$ be a $\pi$-block of a $\pi$-separable group $G$ with defect group $D$. Then $k(B)\le|D|$. 
\end{Thm}
\begin{proof}
We mimic Nagao's reduction~\cite{NagaoKGV} of Brauer's $k(B)$-Conjecture for $p$-solvable groups.
Let $N:=\pcore_{\pi'}(G)$, and let $\lambda\in\Irr(N)$ lying under $B$.
By \cite[Theorem~2.10]{Slattery} and \cite[Corollary~2.8]{Slattery2}, the Fong--Reynolds Theorem holds for $\pi$-blocks. Hence, we may assume that $\lambda$ is $G$-stable and $B$ is the set of irreducible characters of $G$ lying over $\lambda$ (see \cite[Theorem~2.8]{Slattery}). Then $D$ is a Hall $\pi$-subgroup of $G$ by the definition of defect groups. By \cite[Problem~11.10]{Isaacs} and \autoref{cor}, it follows that $k(B)\le k(G/N)\le|G|_\pi=|D|$. 
\end{proof}

In the situation of \autoref{kGV} it is known that $GA$ contains only one $\pi$-block where $\pi$ is the set of prime divisors of $|G|$ (see \cite[Corollary~2.9]{Slattery}). Thus, in the proof of \autoref{main} one really needs to full strength of \autoref{kGV}.

Liu~\cite{Yanjun} has also proposed the following conjecture (cf. \cite[Definition~2.13]{Slattery2}):

\begin{Con}[Olsson's Conjecture for $\pi$-blocks]
Let $B$ be a $\pi$-block of a $\pi$-separable group $G$ with defect group $D$. Let $k_0(B)$ be the number of characters $\chi\in B$ such that $\chi(1)_\pi |D|=|G|_\pi$. Then $k_0(B)\le|D:D'|$.
\end{Con} 

This conjecture however is false. A counterexample is given by $G=\PSL(2,2^5)\rtimes C_5$ where $C_5$ acts as a field automorphism on $\PSL(2,2^5)$. Here $|G|=2^5\cdot 3\cdot 5\cdot 11\cdot 31$ and we choose $\pi=\{2,3,11,31\}$. Then $\pcore_{\pi}(G)=\PSL(2,2^5)$ and \cite[Corollary~2.9]{Slattery} implies that $G$ has only one $\pi$-block $B$ which must contain the five linear characters of $G$. Moreover, $B$ has defect group $D=\pcore_{\pi}(G)$ by \cite[Lemma~2.3]{Slattery2}. Hence, $k_0(B)\ge 5>1=|D:D'|$ since $D$ is simple.

\section{Abelian defect groups}

In this section we prove that the equality $k(B)=|D|$ in \autoref{main} can only hold if $D$ is abelian. We begin with Gallagher's observation~\cite{Gallagher} that $k(G)=k(N)k(G/N)$ for $N\unlhd G$ implies $G=\C_G(x)N$ for all $x\in N$. 
Next we analyze equality in our three results above.

\begin{Lem}\label{lemnonabel}
Let $G$ be a finite group and $A\le\Aut(G)$ such that $(|G|,|A|)=1$. If $k(G\rtimes A)=|G|$, then $G$ is abelian.
\end{Lem}
\begin{proof}
We assume that $k(GA)=|G|$ and argue by induction on $|G|$. Suppose first that there is an $A$-invariant normal subgroup $N\unlhd G$ such that $1<N<G$. As in the proof of \autoref{kGV} we set $B:=\C_A(G/N)$ and obtain
$k(GA)=k(NB)k(GA/NB)$. By induction, $N$ and $G/N$ are abelian and $GA=\C_{GA}(x)NB=\C_{GA}(x)B$ for every $x\in N$. Hence $G\le\C_{GA}(x)$ and $N\le\Z(G)$. Therefore, $G$ is nilpotent (of class at most $2$). 
Then every Sylow subgroup of $G$ is $A$-invariant and we may assume that $G$ is a $p$-group. In this case the claim follows from \cite[Theorem~1']{Robinsonnonabel}.

Hence, we may assume that $G$ is characteristically simple. If $G$ is non-abelian, then we easily get a contradiction by following the arguments in the proof of \autoref{kGV}.
\end{proof}

\begin{Lem}\label{lempipi}
Let $G$ be a $\pi$-separable group such that $\pcore_{\pi'}(G)=1$ and $k(G)=|G|_{\pi}$. Then $G=\pcore_{\pi\pi'}(G)$.
\end{Lem}
\begin{proof}
Let $N:=\pcore_{\pi\pi'}(G)$. Since $\pcore_{\pi'}(N)\le\pcore_{\pi'}(G)=1$, we have $k(N)\le|N|_{\pi}$ by \autoref{cor}. Moreover, $\pcore_{\pi'}(G/N)=1$, $k(G/N)\le|G/N|_{\pi}$ and $k(G)=k(N)k(G/N)$. In particular, $G=\C_G(x)N$ for every $x\in N$. Let $g\in G$ be a $\pi$-element. Then $g$ is a class-preserving automorphism of $N$ and also of $N/\pcore_{\pi}(G)$. Since $N/\pcore_{\pi}(G)=\pcore_{\pi'}(G/\pcore_{\pi}(G))$ is a $\pi'$-group, it follows that $g$ acts trivially on $N/\pcore_{\pi}(G)$. 
By the Hall--Higman Lemma 1.2.3, $N/\pcore_{\pi}(G)$ is self-centralizing and therefore $g\in N$. Thus, $G/N$ is a $\pi'$-group and $N=G$. 
\end{proof}

\begin{Thm}
Let $B$ be a $\pi$-block of a $\pi$-separable group with non-abelian defect group $D$. Then $k(B)<|D|$.
\end{Thm}
\begin{proof}
We assume that $k(B)=|D|$. Following the proof of \autoref{main}, we end up with a $\pi$-separable group $G$ such that $D\le G$, $\pcore_{\pi'}(G)=1$ and $k(G)=|G|_{\pi}=|D|$. By \autoref{lempipi}, $D\unlhd G$ and by \autoref{lemnonabel}, $D$ is abelian.
\end{proof}

Similar arguments imply the following $\pi$-version of \cite[Theorem~3]{Robinsonnonabel} which also extends \autoref{cor}. 

\begin{Thm}
Let $G$ be a $\pi$-separable group such that $\pcore_{\pi'}(G)=1$ and $H\le G$. Then $k(H)\le|G|_{\pi}$ and equality can only hold if $|H|_{\pi}=|G|_{\pi}$.
\end{Thm} 

The proof is left to the reader. 

\section*{Acknowledgment}
This work is supported by the German Research Foundation (projects SA \mbox{2864/1-1} and SA \mbox{2864/3-1}).


\begin{thebibliography}{10}

\bibitem{Brauer46}
R. Brauer, \textit{On blocks of characters of groups of finite order. {II}},
  Proc. Nat. Acad. Sci. U.S.A. \textbf{32} (1946), 215--219.

\bibitem{Atlas}
J.~H. Conway, R.~T. Curtis, S.~P. Norton, R.~A. Parker and R.~A. Wilson,
  \textit{ATLAS of finite groups}, Oxford University Press, Eynsham, 1985.

\bibitem{DixonTour}
J.~D. Dixon, \textit{The maximum order of the group of a tournament}, Canad.
  Math. Bull. \textbf{10} (1967), 503--505.

\bibitem{Gallagher}
P.~X. Gallagher, \textit{The number of conjugacy classes in a finite group},
  Math. Z. \textbf{118} (1970), 175--179.

\bibitem{RobRob}
R.~M. Guralnick and G.~R. Robinson, \textit{On the commuting probability in
  finite groups}, J. Algebra \textbf{300} (2006), 509--528.

\bibitem{INM}
M.~J. Iranzo, G. Navarro and F.~P. Monasor, \textit{A conjecture on the number
  of conjugacy classes in a {$p$}-solvable group}, Israel J. Math. \textbf{93}
  (1996), 185--188.

\bibitem{Isaacs}
I.~M. Isaacs, \textit{Character theory of finite groups}, AMS Chelsea
  Publishing, Providence, RI, 2006.

\bibitem{Laradji}
A. Laradji, \textit{Relative {$\pi$}-blocks of {$\pi$}-separable groups}, J.
  Algebra \textbf{220} (1999), 449--465.

\bibitem{Laradji2}
A. Laradji, \textit{Relative {$\pi$}-blocks of {$\pi$}-separable groups. {II}},
  J. Algebra \textbf{237} (2001), 521--532.

\bibitem{Yanjun}
Y. Liu, \textit{{$\pi$}-forms of {B}rauer's {$k(B)$}-conjecture and {O}lsson's
  conjecture}, Algebr. Represent. Theory \textbf{14} (2011), 213--215.

\bibitem{ManzStaszewski}
O. Manz and R. Staszewski, \textit{Some applications of a fundamental theorem
  by {G}luck and {W}olf in the character theory of finite groups}, Math. Z.
  \textbf{192} (1986), 383--389.

\bibitem{NagaoKGV}
H. Nagao, \textit{On a conjecture of {B}rauer for {$p$}-solvable groups}, J.
  Math. Osaka City Univ. \textbf{13} (1962), 35--38.

\bibitem{PP}
P.~P. P\'alfy and L. Pyber, \textit{Small groups of automorphisms}, Bull.
  London Math. Soc. \textbf{30} (1998), 386--390.

\bibitem{Robinsonnonabel}
G.~R. Robinson, \textit{On {B}rauer's {$k(B)$}-problem for blocks of
  {$p$}-solvable groups with non-{A}belian defect groups}, J. Algebra
  \textbf{280} (2004), 738--742.

\bibitem{kGVproblem}
P. Schmid, \textit{The solution of the {$k(GV)$} problem}, ICP Advanced Texts
  in Mathematics, Vol. 4, Imperial College Press, London, 2007.

\bibitem{Slattery}
M.~C. Slattery, \textit{Pi-blocks of pi-separable groups. {I}}, J. Algebra
  \textbf{102} (1986), 60--77.

\bibitem{Slattery2}
M.~C. Slattery, \textit{Pi-blocks of pi-separable groups. {II}}, J. Algebra
  \textbf{124} (1989), 236--269.

\bibitem{Slattery3}
M.~C. Slattery, \textit{Pi-blocks of pi-separable groups. {III}}, J. Algebra
  \textbf{158} (1993), 268--278.

\bibitem{Wolf}
T.~R. Wolf, \textit{Variations on {M}c{K}ay's character degree conjecture}, J.
  Algebra \textbf{135} (1990), 123--138.

\bibitem{Yixin2}
Y. Zhu, \textit{On {$\pi$}-block induction in a {$\pi$}-separable group}, J.
  Algebra \textbf{235} (2001), 261--266.

\end{thebibliography}
\end{document}